\documentclass{amsart}
  \usepackage{amscd,amssymb,epsfig}
   \usepackage{epic,eepic}

\input{amssym.def}
\input{amssym}
\usepackage{amscd}

                     \numberwithin{equation}{section}

                     \newtheorem{propo}{Proposition}[section]
                     \newtheorem{corol}[propo]{Corollary}
                     \newtheorem{theor}[propo]{Theorem}
                     \newtheorem{lemma}[propo]{Lemma}
                     \theoremstyle{definition}

                     \theoremstyle{remark}
                     \newtheorem{remar}[propo]{Remark}

		     \newcommand{\CC}{\mathbb{C}}
		     
                     \newcommand{\ZZ}{\mathbb{Z}}

                     \newcommand{\Hom}{\operatorname{Hom}}

                     \newcommand{\id}{\operatorname{id}}

\begin{document}
      \title[Dijkgraaf-Witten invariants of surfaces]{Dijkgraaf-Witten invariants of surfaces and projective representations of groups}
                     \author[Vladimir Turaev]{Vladimir Turaev}
                     \address{%
              IRMA, Universit\'e Louis  Pasteur - C.N.R.S., \newline
\indent  7 rue Ren\'e Descartes \newline
                     \indent F-67084 Strasbourg \newline
                     \indent France \newline
                     \indent  and \newline
                       \indent  Department of Mathematics \newline
                         \indent  Indiana University \newline
                           \indent   Rawles Hall,  831 East 3rd St  \newline
    \indent  Bloomington, IN 47405 \newline
                             \indent  USA \newline  }
                     \begin{abstract} 	We compute the Dijkgraaf-Witten invariants of surfaces in terms of projective representations of groups.  As an application we prove that the 
    complex  Dijkgraaf-Witten invariants of surfaces of positive genus are positive integers.

 {\bf AMS Classification} 57R56,  81T45,  20C25     
 
 {\bf Keywords:}     surfaces, Dijkgraaf-Witten invariants, projective representations 

                     \end{abstract}

                     \maketitle

                  \section{Introduction}\label{se1}
		  
	Dijkgraaf and  Witten \cite{dw}  derived  homotopy invariants of  3-manifolds  from  3-dimensional cohomology classes of finite groups. Their construction  provides   examples  of      path integrals reduced  to   finite sums. It extends to  arbitrary dimensions as follows. Fix a field $F$ and let  $F^*=F-\{0\}$ be  the multiplicative group of non-zero elements of $F$. Fix a  finite group $G$ whose order $\#G$ is invertible in $F$.   Pick an Eilenberg-MacLane  CW-space $X$ of type $K(G,1)$ with base point $x\in X$. Consider a closed connected oriented   topological manifold $M$ of dimension $n\geq 1$ with base point $m_0\in M$ and set $\pi =\pi_1(M,m_0)$. Observe that for any group homomorphism $\gamma:\pi \to G$, there is a  mapping $f_\gamma:(M,m_0)\to (X,x)$  (unique up to homotopy)
such that the induced homomorphism $(f_\gamma)_{\#}: \pi \to \pi_1(X,x)=G$ is equal to $\gamma$.  The {\it Dijkgraaf-Witten invariant}  $Z_\alpha (M)\in F$   determined by a  cohomology class $\alpha\in H^n(G; F^*)=H^n(X;F^*)$ is defined by 
\begin{equation}\label{e1} Z_\alpha (M)=(\#G)^{-1} \sum_{\gamma\in \Hom (\pi , G)} \langle (f_\gamma)^* (\alpha), [M]\rangle  \, .\end{equation}
Here $	 \Hom (\pi , G)$ is the (finite) set of all group homomorphisms
$  \pi \to  G $ and  $$\langle (f_\gamma)^* (\alpha), [M]\rangle\in F^*  $$ is the value 
of   $  (f_\gamma)^* (\alpha)\in H^n(M; F^*)$ on the fundamental class $ [M] \in H_n(M; \ZZ)$.
The     addition on the right-hand side of \eqref{e1} is the  addition in $F$.  One may say that 
  $Z_\alpha (M)$ counts the homomorphisms $\gamma:\pi  \to G$ with weights $(\#G)^{-1} \, \langle (f_\gamma)^* (\alpha), [M]\rangle$. In particular, for $\gamma=1$, we have $\langle (f_\gamma)^* (\alpha), [M]\rangle=1_F$, where $1_F\in F$ is the unit of $F$.  Thus,  $\gamma=1$ contributes the summand $(\#G)^{-1}\cdot 1_F$ to $Z_\alpha (M)$.

  It is clear from the definitions that $Z_\alpha (M)$   depends neither on the choice of the base point $m_0\in M$  nor on the choice of the Eilenberg-MacLane space $X$. Moreover,  $Z_\alpha (M)$ depends only on $\alpha$ and the  homotopy  type of $M$.
  For example, if $M$ is simply connected, that is if   $\pi =\{1\}$, then 
   $Z_{\alpha} (M)= (\#G)^{-1}\cdot 1_F$ for all $\alpha$.  If $G=\{1\}$, then $Z_{\alpha} (M)=1_F$ for all $M$.
   	
	In this paper we  study  the case $ n=2$.  	We begin with elementary algebraic preliminaries. If $p\geq 0$ is the characteristic of the field $F$, then the formula 
	$m\mapsto m\cdot 1_F$ for $m\in \ZZ/p\ZZ$, defines an isomorphism of 
	$\ZZ/p\ZZ$ onto the ring $(\ZZ/p\ZZ)\cdot 1_F\subset F$ additively generated by $1_F$.
	We   identify   $\ZZ/p\ZZ$ with $(\ZZ/p\ZZ)\cdot 1_F$ via this isomorphism. 
	In particular, if $p=0$, then $\ZZ=\ZZ\cdot 1_F\subset F$.
	
		  \begin{theor}\label{cor1}  Let  $F$ be    a  field of characteristic $p\geq 0$ and $\alpha\in H^2(G; F^*)$. If $p=0$, then  $Z_{\alpha} (M)\in F$ is a positive integer for all   closed connected  oriented surfaces $M\neq S^2$.  If $p>0$, then  
		  $Z_{\alpha} (M)\in \ZZ/p\ZZ \subset F$   for all   closed connected  oriented surfaces $M $.\end{theor}

	 The case of the sphere $M=S^2$ stays somewhat apart. Since $S^2$ is simply connected, $Z_{\alpha} (S^2)= (\#G)^{-1}\cdot 1_F$ for all $\alpha$.  If $p>0$, then by the assumptions on $G$, the number $\# G$ is invertible in $\ZZ/p\ZZ$ and  
	 $Z_{\alpha} (S^2)\in \ZZ/p\ZZ \subset F$. If   $p=0$, then 
	 $Z_{\alpha} (S^2)$ is not an integer except for $G=\{1\}$.
	
	Applying Theorem~\ref{cor1} to $F=\CC$, we obtain that the complex number $Z_{\alpha} (M) $ is a positive integer for all    surfaces $M\neq S^2$ and all $\alpha\in H^2(G; \CC^*)$.  
	
	A curious feature of Theorem~\ref{cor1} is that its statement uses  only   classical  notions of algebraic topology   while its proof, given below,  is based on ideas and techniques from quantum topology. I do not know how to prove this theorem  using only the standard tools of    algebraic topology.

	Results similar to Theorem~\ref{cor1} are familiar in the study of 3-dimensional  topological quantum field theories (TQFT's), where     partition functions on surfaces compute   dimensions of certain vector spaces associated to the  surfaces. 
It would be interesting to give an interpretation of  
	$Z_{\alpha} (M)$ as the dimension of a vector space naturally associated with $M$. 

The  proof of Theorem~\ref{cor1} is based on a Verlinde-type formula for $Z_{\alpha} (M)$ stated    in terms of projective representations of $G$.  Let, as above, $F$ be a field (not necessarily algebraically closed). 
	Fix a  2-cocycle $c:G\times G\to F^*$ so that for all $g_1, g_2, g_3 \in G$, 
\begin{equation}\label{d1}c(g_1,g_2) \,c(g_1 g_2, g_3)= c(g_1, g_2  g_3) \, c(g_2, g_3)\, .\end{equation} 
We will always assume that $c$ is {\it normalized} in the sense that  $c(g,1)=c(1,g)=1$ for all $g\in G$. (Any cohomology class  in $H^2(G;F^*)$ can be represented by a normalized cocycle.)
A mapping   $\rho:G\to  GL (W)$, where $W$ is a finite dimensional vector space over $F$, is called a {\it $c$-representation of}  $G$
  if $\rho(1)=\id_W$ and $\rho(g_1  g_2) =c(g_1, g_2)\, \rho (g_1) \,\rho (g_2)$  for any $g_1, g_2\in G$.   
The dimension of $W$ over $F$ is denoted $\dim (\rho)$.  Two $c$-representations $\rho :G\to  GL (W )$ and $\rho':G\to  GL (W')$ are  {\it equivalent} if there is   an isomorphism of vector spaces $j:W \to W' $ such that $\rho'(g)=  j \, \rho (g) \, j^{-1}$ for all $g\in G$.
It is clear that    $\dim (\rho)$   depends only on the equivalence class of $\rho$.

A  $c$-representation $\rho:G\to  GL (W)$ is {\it  irreducible}  if $0$ and $W$ are the only vector subspaces of $W$ invariant under $\rho(G)$. It is obvious that a  $c$-representation equivalent to an irreducible one is itself irreducible. 
The set of   equivalence classes of irreducible $c$-representations of $ G $   is denoted $ \widehat G_c$. We   explain below that  $ \widehat G_c$   is a   finite non-empty set.  

The   2-cocycle $c$   represents a cohomology class  $[c]\in   H^2(G; F^*)$.
The theory of $c$-representations of $ G $ depends only on $[c]$.
Indeed, any two  normalized 2-cocycles $c , c':G\times G\to F^*$  representing the same cohomology class satisfy   $$c (g_1, g_2)=c'(g_1, g_2) \,b(g_1) \,b (g_2) \, (b (g_1  g_2) )^{-1}$$ for all  $g_1, g_2\in G$ and a  mapping $b:G \to F^*$ such that $b(1)=1$.   A  $c$-representation $\rho$ of $G$ gives rise to a  $c'$-representation $\rho'$ of $G$  by $\rho'(g) =b(g) \rho (g)$ for   $g\in G$. This establishes a bijection  between $c$-representations and $c'$-representations  and  induces a bijection $\widehat G_c \approx  \widehat G_{c'}$.

 \begin{theor}\label{th1} Let    $F$ be an algebraically closed field such that    $\#G$ is invertible in $F$. For any  normalized 2-cocycle $c:G\times G\to F^*$ and any closed connected oriented surface $M$, 
\begin{equation}\label{e2}Z_{[c]} (M)= (\#G)^{-\chi(M)} \, \sum_{\rho\in   \widehat G_c} 
	   (\dim \rho)^{\chi(M)} \cdot 1_F \,,\end{equation}
	where $\chi(M)$ is the Euler characteristic of $M$.\end{theor}
	
	It is known that the integer $\dim \rho$ divides $\#G$ in $\ZZ$ for any  $\rho\in   \widehat G_c$, see, for instance,  \cite[p.\ 296]{ka}. Therefore $\dim \rho$ is invertible in $F$ so that the expression on the right hand side of \eqref{e2} is well defined for all values of $\chi(M)$.

For $M=S^1 \times S^1$,  Formula \eqref{e2} gives $Z_{[c]} (S^1 \times S^1)=\# \widehat G_c \, ({\text {mod}}\,  p)$, where $p\geq 0$ is the characteristic of $F$. For $M=S^2$, Formula \eqref{e2} can be rewritten as  $\sum_{\rho\in   \widehat G_c} 
	   (\dim \rho)^{2}= \#G\, ({\text {mod}}\,  p)$.
	If $p=0$, this gives 
\begin{equation}\label{new} \sum_{\rho\in   \widehat G_c} 
	   (\dim \rho)^{2}= \#G\,.
	   \end{equation}
	   We show below that   \eqref{new} holds also when $p>0$.

Denote $M_k$   a closed connected oriented surface of genus $k\geq 0$. If $k\geq 1$, then $-\chi(M)=2k-2\geq 0$ and 
$$Z_{[c]} (M_k)=  \sum_{\rho\in   \widehat G_c} 
	   \bigl (\frac{\# G} {\dim \rho}\bigr)^{2k-2} \cdot 1_F $$
	   is a non-empty sum of positive integers times $1_F$. This implies Theorem \ref{cor1} for algebraically closed $F$.  The general case of Theorem \ref{cor1} follows by taking the   algebraic closure of the ground field (this   does not change the Dijkgraaf-Witten invariant).

For  the trivial cocycle $c=1$, we have $[c]=0$ and   $\widehat G=\widehat G_c$ is the set 	of all irreducible (finite-dimensional)  linear representations of $G$ over $F$ considered up to linear equivalence. Clearly,  $\langle (f_\gamma)^* ([c]), [M]\rangle=1\in F^*$ for any closed connected oriented surface $M$ and any  homomorphism  $\gamma $ from  $\pi=\pi_1(M)$   to $G$. Thus,  $Z_0(M)=
	(\#G)^{-1} \# \Hom (\pi , G)$.   Formula  \eqref{e2} gives  	\begin{equation}\label{e3}\# \Hom (\pi , G)=   \#G \, \, \sum_{\rho\in  \widehat G} 
	   \bigl (\frac{\#G}{ \dim \rho}\bigr )^{-\chi(M)} \, ({\text {mod}} \, p) \, ,\end{equation}
	   where $p\geq 0$ is the characteristic of $F$.
For $M=S^1\times S^1$ and $F=\CC$, Formula \eqref{e3} is  due to Frobenius.   For surfaces of higher genus, this formula was 
	first pointed out by Mednykh \cite{me}, see \cite{jo} for a direct algebraic proof and \cite[Section 5]{fq}   for a proof based  on topological field theory.
	
		One can deduce further properties of    $Z_{[c]} (M)$ from Theorem \ref{th1}. Assume that $F$ is an algebraically closed field of characteristic 0 and $k $ is a positive integer.   By \eqref{new}, we have  ${(\# G)}^{1/2}\geq \dim \rho \geq 1 $ for all $\rho\in  \widehat G_c$ and therefore    $$
	 \# \widehat G_c  \,  {(\# G)}^{2k-2}\geq  	Z_{[c]} (M_k)\geq  \# \widehat G_c  \,  {(\# G)}^{k-1}\, .$$ If $[c]\neq 0$, then $  \dim \rho \geq  2$ for all $\rho\in  \widehat G_c$ and
	 we obtain a more precise estimate from above $\# \widehat G_c  \,  {(\# G/2)}^{2k-2}\geq  	Z_{[c]} (M_k)$.
	 
	 If $\#G=q^N$, where $q$ is a prime integer and $N\geq 1$,   then $Z_{[c]} (M_k)\in \ZZ$ is  divisible by $q^{ N/2 }$  for even $N$   and by $q^{(N+1)/2 }$ for odd $N$.

	 Formula \eqref{e2} is suggested by the fact that    the Dijkgraaf-Witten invariant   (in an arbitrary dimension $n$)  can be extended to an $n$-dimensional   TQFT, see   \cite{dw}, \cite{wa}, \cite{fq}, \cite{fr}. Two-dimensional TQFT's  are known to arise from semisimple algebras, see  \cite{fhk}. We show that  the 2-dimensional Dijkgraaf-Witten invariant $Z_{[c]}$  arises from  the $c$-twisted group algebra of $G$. Then we split  this algebra   as a direct sum of matrix algebras numerated by the elements of $\widehat G_c$ and   use a computation of the state sum invariants of surfaces derived from matrix algebras. Note that the proof of   \eqref{e2}, detailed in   Sections \ref{se2}  and \ref{se3},    actually does not use the notion of a TQFT.  A version of these results for non-orientable surfaces is   discussed in Section \ref{se4}.
	
	The author is indebted to Noah Snyder for a useful discussion of the non-orientable case.

	 \section{State sum invariants of surfaces}\label{se2}

	 We recall   the state sum invariants of oriented surfaces associated with semisimple algebras, see  \cite{fhk}.
	 
	 Let ${\mathcal A}$ be a finite-dimensional algebra over a field $F$. For   $a\in {\mathcal A}$, let $T(a)\in F$ be the trace of the $F$-linear homomorphism ${\mathcal A}\to {\mathcal A}$ sending any $x\in {\mathcal A}$ to $ax$. The resulting $F$-linear mapping $T:{\mathcal A} \to F$ is called the {\it trace homomorphism}. Clearly, $T(ab)=T(ba)$ for all $a,b \in {\mathcal A}$.  Therefore the bilinear form    $T^{(2)}:{\mathcal A}\otimes {\mathcal A} \to F$, defined by $ T^{(2)}(a\otimes b)= T(ab)$  is   symmetric. Assume      that   ${\mathcal A}$ is semisimple in the sense that the form  $T^{(2)}$  is non-degenerate.  We use the adjoint isomorphism ${\rm {ad}}\, T^{(2)} $  to identify   ${\mathcal A}$ with the dual vector space $\Hom_F({\mathcal A},F)$.  Using this identification and dualizing    $T^{(2)}:{\mathcal A}\otimes {\mathcal A} \to F$, we obtain a homomorphism $ F\to  {\mathcal A}\otimes {\mathcal A}$.  It  sends $1\in F$ to a vector $$v=\sum_i v^1_i \otimes v^2_i \in {\mathcal A}\otimes {\mathcal A}\, ,$$ where $i$ runs over a finite set of indices. The vector  $v=v({\mathcal A})$ is 
	 uniquely defined by the identity
	\begin{equation}\label{plom} T (ab)=\sum_i T(av^1_i) \,T( b v^2_i )\end{equation}  for all $a,b\in {\mathcal A}$. The vector $v$  is symmetric in the sense that $ \sum_i v^1_i \otimes v^2_i = \sum_i v^2_i \otimes v^1_i$. 
	 
	 Consider a closed connected oriented surface $M$ and fix a triangulation of $M$. 
	 We endow all   2-simplices of  the triangulation of   $M$ with distinguished orientation induced by the one in $M$.
A {\it flag} of $M$ is a pair (a 2-simplex of the  triangulation, an edge of this 2-simplex). Let    $\{{\mathcal A}_f\}_{f }$ be a set of   copies of ${\mathcal A}$ numerated by all flags $f$ of $M$.  Every edge $e$ of  (the triangulation of) $M$ is incident to two 2-simplices    $\Delta , \Delta'  $ of $M$. Let $v_e\in {\mathcal A}_{(\Delta , e)}\otimes {\mathcal A}_{(\Delta' , e)}$ be a copy of  $v\in {\mathcal A}\otimes {\mathcal A}$. The symmetry of $v$  ensures  that  $v_e$ is well-defined.
 Set $V=\otimes_e v_e\in  \bigotimes_{f } {\mathcal A}_f$, where $e$ runs over all edges of $M$ and $f$ runs over all flags of $M$.

We say that a trilinear form  $U:{\mathcal A}\otimes {\mathcal A} \otimes {\mathcal A} \to F $ is 	{\it cyclically symmetric} if 
$$U (a \otimes b \otimes c)=  U (c\otimes a  \otimes b)   $$
   for all $a, b, c \in {\mathcal A}$.  Then  $U $
   induces a homomorphism $\widetilde U:\bigotimes_{f } {\mathcal A}_f\to F$ as follows. 
   Every 2-simplex $\Delta$ of $M$ has three edges $e_1,e_2,e_3$   numerated  so that following along the  boundary of $ \Delta$ in the direction determined by the distinguished orientation of  $\Delta$, one meets  consecutively   $e_1,e_2, e_3$. Since ${\mathcal A}_{(\Delta, e_i)}= {\mathcal A}$ for $i=1, 2, 3$, the   form $U $ induces a 
	  trilinear form 
	  $$U_\Delta: {\mathcal A}_{(\Delta, e_1)} \otimes {\mathcal A}_{(\Delta, e_2)} \otimes {\mathcal A}_{(\Delta, e_3)}\to F.$$
	  This  form   is cyclically symmetric and therefore independent of the numeration of the edges of $\Delta$. The tensor product  $\otimes_\Delta \,U_\Delta$ over all 2-simplices $\Delta$ of $M$ is 
	  a homomorphism $ \bigotimes_{f } {\mathcal A}_f\to F$ denoted $\widetilde U$.

The    form  $T^{(3)}:{\mathcal A}\otimes {\mathcal A} \otimes {\mathcal A} \to F $  
   sending $a \otimes b \otimes c$ to $T(abc)$ for all $a,b,c \in {\mathcal A}$ is cyclically symmetric.   Consider the induced  homomorphism $\widetilde T^{(3)}: \bigotimes_{f } {\mathcal A}_f\to F$ and set   $I_{\mathcal A}(M)=\widetilde T^{(3)} (V)\in F$. The key property of $I_{\mathcal A}(M)$ is the independence of the choice of triangulation of $M$. It is verified by checking the invariance of $I_{\mathcal A}(M)$ under the Pachner moves on the triangulations.

The direct product and the tensor product of two semisimple algebras are semisimple algebras.
The invariant $I_{\mathcal A}(M)$ is additive with respect to direct products of algebras and multiplicative with respect to tensor products of algebras.

     \begin{lemma}\label{le1}  We have $I_{\mathcal A}(S^2)=  (\dim_F {\mathcal A})\cdot 1_F$. \end{lemma}
    
    \begin{proof} To compute $I_{\mathcal A}(S^2)$, we use the technique of skeletons   (this technique will not be used elsewhere in this paper).  A {\it skeleton of} a   surface  $M$ is a finite  graph embedded in $M$ whose   complement in $M$ consists of open 2-disks. A triangulation of $M$ gives rise to a skeleton of $M$ whose vertices are the barycenters of the 2-simplices of the  triangulation and whose edges are dual to the edges of the triangulation. One can rewrite the definition of $I_{\mathcal A}(M)$ in terms of state sums on    skeletons, see \cite{fhk}, \cite{tu}.  Namely, one assigns  $v=\sum_i v^1_i \otimes v^2_i$ to each edge of the skeleton meaning that  the index $i$ is assigned to the edge, the element $v^1_i$ of ${\mathcal A}$ is assigned to one half-edge and $v^2_i$ to the other half-edge. For each vertex of the skeleton, one   cyclically multiplies the elements of ${\mathcal A}$ assigned in this way  to all incident half-edges and evaluates $T$ on this product. These   values of $T$   are   multiplied over all vertices of the skeleton and the results are summed over all the indices $i$ sitting on the edges.  The  resulting sum is equal to $I_{\mathcal A}(M)$.
    
The 2-sphere $S^2$ has a skeleton $S^1\subset S^2$ having one vertex and one edge. This gives  $I_{\mathcal A}(S^2)= \sum_i T( v^1_i v^2_i)$. To compute the latter expression, we can assume that  the vectors $\{v^1_i\}_{i }$ in the expansion $v=\sum_i v^1_i \otimes v^2_i$  form a basis of  ${\mathcal A}$. Formula \eqref{plom} implies that 
    $$T^{(2)} (a, b)= T^{(2)} (a, \sum_i  T(b v^2_i)   v^1_i)\, ,$$ where $T^{(2)}(a,b)=T(ab)$  for  $a,b\in {\mathcal A}$. Since the bilinear form $T^{(2)}$  is non-degenerate,   $b=\sum_i   T(b v^2_i)   v^1_i$ for all $b\in {\mathcal A}$. For $b=v^1_j$, this gives  $v^1_j=  \sum_i   T(v^1_j v^2_i)   v^1_i$. Since   $\{v^1_i\}_{i }$ is a basis of ${\mathcal A}$, we have $T(v^1_j v^2_i)=1$ if $i=j$ and $T(v^1_j v^2_i)=0$ if  $i\neq j$. 
    
  The trace of any  $F$-linear homomorphism $f:{\mathcal A}\to {\mathcal A}$ can be expanded via the trace homomorphism $T:{\mathcal A}\to F$ as follows. Represent $f$ by a   matrix  $(f_{i,j})$  over $F$  in the basis 
  $\{v^1_i\}_{i }$ so that   $f(v^1_i)=\sum_j f_{i,j} v^1_j$ for all $i$. Then  $${\rm {Tr}}(f)
  =\sum_i f_{i,i}= \sum_{i,j} f_{i,j} T(v^1_j v^2_i)=\sum_i T(f(v^1_i) v^2_i).$$
Pick $a\in {\mathcal A}$ and  consider  the homomorphism $f_a:{\mathcal A}\to {\mathcal A}$ sending any $x\in {\mathcal A}$ to $ax$.
By the previous formula,  $$T^{(2)} (a, 1_{\mathcal A})=T(a)={\rm {Tr}} (f_a)=T(  \sum_i  a v^1_i  v^2_i) =T^{(2)} (a,  \sum_i    v^1_i  v^2_i) \,,$$  
where $1_{\mathcal A}$ is the unit of ${\mathcal A}$. The non-degeneracy 
of $T^{(2)}$ implies that $\sum_i  v^1_i  v^2_i=1_{\mathcal A}$. Thus, 
$$I_{\mathcal A}(S^2) = T(\sum_i  v^1_i  v^2_i)=T(1_{\mathcal A})=  (\dim_F {\mathcal A})\cdot 1_F\, .$$
 \end{proof}
     
{\bf Example.} Let ${\mathcal A}={\rm {Mat}}_{d } (F)$ be  the algebra of $(d  \times d )$-matrices over   $F$ with $d\geq 1$. This algebra is semisimple if and only if   $d$ is invertible in $F$ and then  
$I_{ {\mathcal A}} (M)=d^{\chi (M)}\cdot 1_F$  for any closed connected oriented surface $M$, see \cite[Theorem 4.2]{sn}. For $d=1$, we obtain $I_{ F} (M)=  1_F$ for all  $M$. 
 
	 \section{Invariants derived from twisted group algebras}\label{se3}

	  Let $G$ be a group and $ F[G]$ be the vector space over a field $F$ with basis $G$. A  normalized 2-cocycle $c:G\times G\to F^*$ gives rise to a multiplication law $\cdot  $ on $F[G]$   by $	 g_1\cdot  g_2= c(g_1, g_2) \, g_1 g_2$, where $g_1,g_2$ run over $G$ and $g_1g_2\in G$ is the  product in $G$. The vector space $ F[G]$ with this multiplication is an associative algebra and  the neutral element  $1\in G \subset F[G]$ is its unit. This algebra is  called the {\it twisted group algebra of} $G$ and denoted $A^{(c)}$.
It is easy to check that  the isomorphism type of $A^{(c)}$ depends only on the cohomology class  $[c]\in H^2(G;F^*)$.

From now on,    $G$ is a finite group whose order $\#G$ is invertible in   $F$. The  algebra $A^{(c)}$ is $(\#G)$-dimensional and  the trace homomorphism $T:A^{(c)} \to F$ defined in Section \ref{se2} sends $1\in G $ to $\#G$ and sends all other basis vectors of $A^{(c)}$ to zero.  The associated bilinear form    $T^{(2)}:A^{(c)}\otimes A^{(c)} \to F$ sends a  pair of basis vectors $g_1,g_2\in G$
to $\#G$  if $g_2=g_1^{-1}$ and to 0 otherwise. This form is non-degenerate and so the algebra $A^{(c)}$ is semisimple. The following theorem shows that the state sum  invariant of surfaces $I_{A^{(c)}}$  is equivalent to    the Dijkgraaf-Witten invariant derived from   $[c]\in H^2(G;F^*)$.
 
    \begin{theor}\label{bdbdbdad}  For  any normalized 2-cocycle $c:G\times G\to F^*$ on   $G$ and any closed connected oriented surface $M$,   $$Z_{[c]} (M)=(\#G)^{-\chi(M)} \, I_{A^{(c)}}(M)\,.$$\end{theor}
    
    \begin{proof} Fix a triangulation of $M$ and let $k_0$, $k_1$, $k_2$ be respectively  the number of vertices, edges, and 2-simplices of this triangulation.    By an {\it oriented edge} of $M$ we mean an edge of (the triangulation of)  $M$ endowed with an arbitrary orientation.  For an  oriented edge $e$ of $M$,   the same edge with opposite orientation is denoted $-e$.  A {\it labeling} of $M$ is a mapping  $\ell$  from the set of oriented edges of    $M$ to $G$ such that $\ell (-e)=(\ell (e))^{-1}$ for all oriented edges $e$ of $M$.     A labeling $\ell$ of $M$ is   {\it admissible} if  
    $  \ell(e_1)\, \ell(e_2)\, \ell(e_3)=1$  
for any    three consecutive oriented edges  $e_1,e_2,e_3$ forming the boundary of a 2-simplex  of (the  triangulation of) $M$. Denote the set of labelings of $M$ by $L(M)$ and denote its subset formed by the admissible labelings by $L_a(M)$.
   
  Given a labeling $\ell\in L (M)$, we   assign to any path $p$ in $M$ formed by consecutive oriented edges $e_1,  \ldots, e_N$ the product  $ \ell (p)=   \ell (e_1) \ell (e_2) \cdots \ell (e_N) \in G$. For  admissible $ \ell $, this product    is a homotopy invariant of $p$:  if two   paths $p,p'$ have the same endpoints and are homotopic (relative to the endpoints), then $ \ell(p)= \ell (p')$.

  Fix a base vertex $m_0\in M$ and set $\pi=\pi_1(M,m_0)$. For any $\ell\in L_a(M)$, applying the mapping $p\mapsto   \ell (p)$ to the loops in $M$ based at $m_0$, we obtain   a group homomorphism $ \pi \to G$ denoted $\Gamma (\ell)$.   The formula $\ell \mapsto \Gamma (\ell)$ defines a   mapping   $\Gamma:L_a (M)\to \Hom (\pi, G) $.

  We claim that  the pre-image $\Gamma^{-1}(\gamma)$ of any $\gamma\in \Hom (\pi, G)$ consists of   $(\#G)^{k_0-1}$ admissible labelings. To see this, 
fix a spanning tree $R\subset M$ formed by all $k_0$ vertices and $k_0-1$ edges of $M$; here we use that $M$ is connected. For every vertex $m$ of $M$, there is a  (unique up to homotopy)  path $p_m$ in $R$ formed by oriented edges of $R$ and leading from   $m_0 $ to $m$. Any oriented edge $e$ of $M$ not lying in $R$ determines a loop $p_{s_e} e\, (p_{t_e})^{-1}$, where $s_e$ and $t_e$ are    the initial and the terminal endpoints of $e$, respectively.   The homotopy classes of such  loops  corresponding to all oriented edges $e$ of $M$ not lying in $R$ generate the fundamental group $\pi $. Therefore  the pre-image of   $\gamma\in \Hom (\pi, G)$ consists of the  labelings $\ell \in L_a(M)$  such that $\ell (p_{s_e} e \, (p_{t_e})^{-1})=\gamma (p_{s_e} e \,(p_{t_e})^{-1})$ for all   $e$ as above. This equality may be rewritten as
\begin{equation}\label{49}\ell (e)=(\ell (p_{s_e}))^{-1} \,\, \gamma (p_{s_e} e\,  (p_{t_e})^{-1}) \,\, \ell (p_{t_e})\,. \end{equation} Therefore to specify $\ell\in \Gamma^{-1} (\gamma)$, we can assign arbitrary labels   to the  $k_0-1$ edges    of $R$   oriented away from $m_0$,  the inverse labels to    the  same edges       oriented towards $m_0$,  and    the labels determined from \eqref{49} to   the  oriented edges of $M$ not lying in $R$. The resulting labeling is necessarily admissible.  Hence, $\#\, \Gamma^{-1}(\gamma)= (\#G)^{k_0-1}$.

Formula \eqref{e1} and the results above imply that 
$$ Z_{[c]} (M)=(\#G)^{-k_0}\sum_{\ell \in L_a(M)} \langle (f_{\Gamma(\ell)})^* ([c]), [M]\rangle \in F\, ,$$
where   $f_{\Gamma(\ell)}$ is a mapping from the pair $(M,m_0)$ to the pair (an Eilenberg-MacLane space $X$ of type $K(G,1)$, a base point $x\in X$)  
such that the induced homomorphism of fundamental groups   is equal to $\Gamma(\ell): \pi \to  G$.     
Choosing in the role of $X$ the canonical realization of the Eilenberg-MacLane space $K(G,1)$ associated with the standard resolution of the $\ZZ[G]$-module $\ZZ$ (see, for instance, \cite{br}), we can  compute   $\langle (f_{\Gamma(\ell)})^* ([c]), [M]\rangle$ as follows.
Fix a total order $<$ on the set of all vertices of $M$. A 2-simplex $\Delta$ of $M$ has three vertices $A,B, C$ with $A<B<C$. Set $\varepsilon_\Delta=+1$ if the distinguished orientation of   $\Delta$   (induced by the one on $M$) induces  the direction from $A$ to $B$ on the edge $AB\subset \partial \Delta$ and set $\varepsilon_\Delta=-1$ otherwise. Let $\ell_1^\Delta=\ell (AB)$ and $ \ell_2^\Delta= \ell (BC)$  be the labels of the edges $AB$,  $BC$ oriented  from $A$ to $B$ and from $B$ to $C$, respectively. Then 
$$\langle (f_{\Gamma(\ell)})^* ([c]), [M]\rangle=\prod_\Delta \, c(\ell_1^\Delta,\ell_2^\Delta)^{\varepsilon_\Delta}\in F^*\,,$$
where $\Delta$ runs over all 2-simplices of   $M$.
Hence, 
\begin{equation}\label{e491} Z_{[c]} (M)=(\#G)^{-k_0}\sum_{\ell \in L_a(M)} \, \prod_\Delta \,c(\ell_1^\Delta,\ell_2^\Delta)^{\varepsilon_\Delta} \, .\end{equation}

  We now compute   $I_{{\mathcal A}}(M)\in F$ for ${\mathcal A}=A^{(c)}$. First, with each labeling $\ell\in L(M)$ we associate  an element $\langle c, \ell \rangle $ of $F^*$. Observe  that $c(g,g^{-1})= c(g^{-1}, g)$ for all $g\in G$ (this is obtained from \eqref{d1} by the substitution $g_1=g_3=g$ and $ g_2=g^{-1}$). Therefore
   for any oriented edge $e$ of $M$, the expression $c(\ell(e), \ell(-e))=c(\ell(-e), \ell(e))\in F^*$ does not depend on the orientation of $e$ and may be associated with  the underlying unoriented edge. 
   Set
   $$\langle c, \ell \rangle =\prod_e c(\ell(e), \ell(-e))\in F^*\, ,$$
   where $e$ runs over all non-oriented edges of $M$.

   Let as above $\{{\mathcal A}_f\}_{f }$ be a set of   copies of ${\mathcal A}$ numerated by all flags $f$ of $M$. With 
a   labeling $\ell\in L(M)$ we associate  a vector  
   $V(\ell)\in \otimes_f \, {\mathcal A}_f$ as follows.
   For a flag $f$ formed by a 2-simplex $\Delta$ and its edge $e$, the distinguished orientation of  $\Delta$ induces an orientation of $e$. Let $\ell (f)\in G\subset {\mathcal A}={\mathcal A}_f$ be the value of $\ell$ on this oriented edge. Set
   $V(\ell)= \otimes_f \, \ell (f)$. 
   
Recall the homomorphisms $T:{\mathcal A}\to F$,  $T^{(3)}:{\mathcal A}\otimes {\mathcal A} \otimes {\mathcal A} \to F $, and the vector $v\in {\mathcal A} \otimes {\mathcal A} $ introduced in Section \ref{se2}. It  is easy to compute that  $$v=(\# G)^{-1} \sum_{g\in G} (c(g,g^{-1}))^{-1} g \otimes g^{-1}\, .$$
Therefore the vector $V=\otimes_e v_e\in \otimes_f \, {\mathcal A}_f$    is computed by 
   $$V= (\#G)^{-k_1}\sum_{\ell \in L(M)}  \langle c, \ell \rangle^{-1} \, V(\ell)\, .$$
 
For any $g_1, g_2, g_3 \in G$, we have  $ T^{(3)} (g_1\otimes g_2 \otimes g_3)=0$ if $g_1g_2g_3 \neq 1$
   and $$ T^{(3)} (g_1\otimes g_2 \otimes g_3)=T(g_1g_2g_3)= \#G \,c(g_1,g_2) \, c(g_1 g_2, g_3)$$ if $g_1g_2g_3=1$. Consider the homomorphism $U=(\# G)^{-1}\,  T^{(3)}: {\mathcal A}\otimes {\mathcal A} \otimes {\mathcal A} \to F $ sending 
   $g_1\otimes g_2 \otimes g_3$ to $0$ if $g_1g_2g_3 \neq 1$
and to $c(g_1,g_2) \, c(g_1 g_2, g_3)$ if $g_1g_2g_3=1$.   The cyclic symmetry of $T^{(3)}$ implies that $U$ is cyclically symmetric.
  Then
   $$I_{\mathcal A}(M)=\widetilde T^{(3)} (V)=(\#G)^{k_2} \,\widetilde U (V)   = (\#G)^{k_2-k_1} \sum_{\ell \in L(M)} \langle c, \ell \rangle^{-1} \,\widetilde U (V(\ell))\, .$$
   It is clear that $k_2-k_1=\chi(M)-k_0$ and $\widetilde U(V(\ell))=0$ for non-admissible $\ell$. Hence \begin{equation}\label{eq5005} I_{\mathcal A}(M) = (\#G)^{\chi(M)-k_0} \sum_{\ell \in L_a(M)} \langle c, \ell \rangle^{-1} \,\widetilde U (V(\ell)) \, .\end{equation}
  
 To compute  $ \widetilde U (V(\ell))$ for   $\ell\in L_a(M)$,  we use the  total order $<$ on the set of  vertices of $M$. With a 2-simplex $\Delta=ABC$ of $M$ with $A<B<C$ we associated above a sign $ \varepsilon_\Delta=\pm 1$ and two  labels   $g_1=\ell_1^\Delta =\ell (AB)\in G$ and $g_2= \ell_2^\Delta=\ell (BC)\in G$. 
 Set also $g_3=\ell_3^\Delta  =\ell (CA)\in G$. The admissibility of $\ell$ implies that  $g_1   g_2
 g_3=1$. 
 The 2-simplex  $\Delta$ gives rise to the  flags $(\Delta, AB)$, $(\Delta, BC)$, and $(\Delta, CA)$. These flags contribute  to   $V(\ell)$ the tensor factor 
 $$(g_1)^{\varepsilon_\Delta} \otimes  (g_2)^{\varepsilon_\Delta} \otimes   (g_3)^{\varepsilon_\Delta}  \in {\mathcal A}_{(\Delta, AB)} \otimes {\mathcal A}_{(\Delta, BC)} \otimes {\mathcal A}_{(\Delta, CA)}  .$$
 Recall the trilinear form  $U_\Delta  :{\mathcal A}_{(\Delta, AB)} \otimes {\mathcal A}_{(\Delta, BC)} \otimes {\mathcal A}_{(\Delta, CA)} \to F$ introduced in Section \ref{se2}. 
If $\varepsilon_\Delta =+ 1$, then   $$ U_\Delta ((g_1)^{\varepsilon_\Delta} \otimes  (g_2)^{\varepsilon_\Delta} \otimes   (g_3)^{\varepsilon_\Delta})=U(g_1\otimes g_2 \otimes g_3)
$$
$$= c(g_1,  g_2 ) \, c(g_1   g_2  ,  g_3 )=c(g_1,  g_2 )
 \, c(g_3^{-1} ,  g_3 )=c(g_1,  g_2 )
 \, c(  g_3 , g_3^{-1} ).$$
If $\varepsilon_\Delta =- 1$, then   $$
U_\Delta ((g_1)^{\varepsilon_\Delta} \otimes  (g_2)^{\varepsilon_\Delta} \otimes   (g_3)^{\varepsilon_\Delta})=U_\Delta (g_1^{-1} \otimes g_2^{-1} \otimes  g_3^{-1})= U(g_3^{-1} \otimes g_2^{-1} \otimes g_1^{-1})$$
$$ = c(g_3^{-1},  g_2^{-1}) \, c(g_3^{-1}g_2^{-1}  ,  g_1^{-1}) =c(g_1   g_2,   g_2^{-1})
 \, c(   g_1, g_1^{-1} )  $$
$$ =(c(g_1,   g_2 ))^{-1} 
 \, c(   g_1, g_1^{-1} )\, c(   g_2, g_2^{-1} ) \, .$$
The last equality follows from   \eqref{d1}, where we set $g_3=g_2^{-1}$. 
In both cases 
 $$U_\Delta ((g_1)^{\varepsilon_\Delta} \otimes  (g_2)^{\varepsilon_\Delta} \otimes   (g_3)^{\varepsilon_\Delta})=c(g_1, g_2)^{\varepsilon_\Delta}  \, u_\Delta \, ,$$
 where $u_\Delta =c(g_3 ,  g_3^{-1} )$ if  $\varepsilon_\Delta =+ 1$ and $u_\Delta=c(   g_1, g_1^{-1} )\, c(   g_2, g_2^{-1} )$  if $\varepsilon_\Delta =- 1$.
 We   conclude that  
\begin{equation}\label{eq555}\widetilde U (V(\ell))=  
\prod_\Delta c(\ell_1^\Delta,\ell_2^\Delta)^{\varepsilon_\Delta} \, u_\Delta\, ,\end{equation}
where $\Delta$ runs over all 2-simplices of $M$.

We claim that  $\prod_\Delta u_\Delta  =\langle c, \ell \rangle $. Note that the product $\prod_\Delta u_\Delta$ expands as a product of the expressions $c(\ell (e), \ell (-e) )$ associated with edges $e$ of $M$. We show that every edge  $e=AB$ of $M$ with $A<B$ contributes exactly one such expression. Set  $g=\ell (AB)\in G$.   The edge $AB$ is incident to two 2-simplices $\Delta =ABC $ and  $\Delta'=ABC'$ of $M$ whose distinguished orientations  induce  on $AB = \partial \Delta \cap \partial \Delta'$ the directions from $B$ to $A$ and  from $A$ to $B$, respectively. If $B<C $, then  $\varepsilon_\Delta =-1$, $g=\ell_1^{\Delta }$,  and   $AB$ contributes   the factor $c(g, g^{-1})$ to $u_\Delta$. 
If $ C <A$, then   $\varepsilon_\Delta =-1$, $g=\ell_2^{\Delta}$,  and $AB$ contributes   the factor $c(g, g^{-1})$ to $u_\Delta$. Finally, if $A<C<B$, then 
 $\varepsilon_\Delta=+1$, $g=\ell_3^{\Delta}$,  and $u_\Delta=c(g, g^{-1})$. A similar computation shows that $AB$ contributes no factors to
  $u_{\Delta'}$. Therefore $$ \prod_\Delta u_\Delta= \prod_e c(\ell(e), \ell(-e))=\langle c, \ell \rangle \, .$$
Substituting this   in \eqref{eq555},  we obtain
$$\widetilde U (V(\ell))= 
\prod_\Delta c(\ell_1^\Delta,\ell_2^\Delta)^{\varepsilon_\Delta}\,\times \,  \langle c, \ell \rangle \,.$$
Formula \eqref{eq5005} yields 
$$ I_{\mathcal A}(M) = (\#G)^{\chi(M)-k_0} \sum_{\ell \in L_a(M)} \prod_\Delta c(\ell_1^\Delta,\ell_2^\Delta)^{\varepsilon_\Delta}  .$$
Comparing   with \eqref{e491}, we obtain the claim of the theorem. \end{proof}

 Since the   algebra $ A^{(c)}$ is finite dimensional and semisimple, the   isomorphism classes of simple $ A^{(c)}$-modules form a  finite non-empty  set $\Lambda$ (an   $ A^{(c)}$-module is {\it simple} if its only $ A^{(c)}$-submodules are itself and zero).  Let $\{V_\lambda\}_{\lambda \in \Lambda}$ be representatives of these isomorphism classes.  Then  \begin{equation}\label{dxd} A^{(c)}\cong \bigoplus_{\lambda \in \Lambda} {\rm {Mat}}_{d_\lambda} (D_\lambda) =  \bigoplus_{\lambda \in \Lambda} D_\lambda \otimes_F {\rm {Mat}}_{d_\lambda} (F) \, ,\end{equation}
 where   $D_\lambda={\rm {End}}_{A^{(c)}} (V_\lambda)$ 
is a division $F$-algebra (i.e., an $F$-algebra  in which all non-zero elements are invertible), $d_\lambda$ is the dimension of $V_\lambda$ as a $D_\lambda$-module, and   ${\rm {Mat}}_{d } (D)$ with $d\geq 1$ is the algebra of $(d  \times d )$-matrices over the ring $D$. The integer $d_\lambda$ is  invertible in $F$ for all $\lambda$, because the algebra $ {\rm {Mat}}_{d_\lambda} (D_\lambda)$ is a direct summand of $A^{(c)}$ and is therefore semisimple. Theorem \ref{bdbdbdad} implies that (under the conditions of this theorem)
\begin{equation}\label{rop} Z_{[c]} (M)=(\#G)^{-\chi(M)} \, \sum_{\lambda \in \Lambda} \,  I_{D_\lambda} (M) \,d_{\lambda}^{\chi (M)}\,.\end{equation}

We can check \eqref{rop}   directly for   $M=S^2$. Indeed, $Z_{[c]} (S^2)= (\#G)^{-1} \cdot 1_F$
and by Lemma \ref{le1}, $I_{D_\lambda} (S^2)=(\dim_F D_\lambda) \cdot 1_F$ for all $\lambda\in \Lambda$. Therefore
  \eqref{rop} for $M=S^2$  follows from the  equality  \begin{equation}\label{rop+-}\#G=\sum_{\lambda \in \Lambda} (\dim_F D_\lambda) \, d_{\lambda}^{2}\end{equation}  which   is a consequence of  \eqref{dxd}.
  
  Formulas \eqref{rop}  and \eqref{rop+-} simplify in the case  where   $F$ is   algebraically closed. Then $D_\lambda=F$ for all $\lambda\in \Lambda$ and   we obtain 
\begin{equation}\label{rop+}Z_{[c]} (M)=(\#G)^{-\chi(M)} \, \sum_{\lambda \in \Lambda}  d_{\lambda}^{\chi (M)}\cdot 1_F \,\end{equation}
and
\begin{equation}\label{rop+-+}\#G=\sum_{\lambda \in \Lambda}   \, d_{\lambda}^{2} \, .\end{equation}
In particular, $Z_{[c]} (S^1\times S^1)=\# \Lambda\cdot 1_F $. The number $\# \Lambda$, that is the number of 
  isomorphism classes of simple $ A^{(c)}$-modules, can be   computed in terms of so-called $c$-regular classes of $G$, see \cite[pp.\ 107-118]{ka}. An element $g\in G$ is {\it $c$-regular}
  if $c(g,h)=c(h,g)$ for all $h\in G$ such that $gh=hg$.  The set of $c$-regular elements of $G$ depends only on the cohomology class $[c]\in H^2(G;F^*)$ and is invariant under conjugation  in  $G$.     Since $F$ is   algebraically closed  (and as always in this paper, $\#G$ is invertible in $F$),    $\# \Lambda=r(G;c)$, where $r(G;c)$ is the number of conjugacy classes of $c$-regular elements of $G$, see \cite[p.\  117]{ka}.

{\bf Proof of Theorem \ref{th1}}.   Any $c$-representation  of $G$  in the sense of Section \ref{se1} extends by linearity to an action of $ A^{(c)}$ on the corresponding vector space.  This gives   a  bijective correspondence between  $c$-representations of $G$
 and $ A^{(c)}$-modules of finite dimension over $F$. This correspondence transforms equivalent representations  to isomorphic $ A^{(c)}$-modules and    irreducible representations  to simple modules. This allows us to rewrite all the statements of this  section in terms of 
the $c$-representations of $G$. In particular, the set   $ \widehat G_c$ of   equivalence classes of irreducible $c$-representations of $ G $   is finite and non-empty. 
 Now, Formula \eqref{rop+} directly implies   \eqref{e2}. Note also that Formula \eqref{rop+-+} implies \eqref{new}. 
 
 \begin{remar} Formula \eqref{e3} generalizes to surfaces with boundary as follows (see \cite{me2},   \cite{jo}). Let $\pi$ be the fundamental group of a compact connected oriented surface $M$ whose boundary  consists of $k\geq 1$ circles and let $x_1, \ldots, x_k$ be the conjugacy classes in $\pi$ represented by the components of $\partial M$. For any   $g_1,\ldots, g_k\in G$, the number of homomorphisms $\varphi:\pi\to G$ such that $\varphi(x_i)$ is conjugate to $g_i$ in $G$ for all $i=1,\ldots , k$ is equal to  $$ \#G \, \, \sum_{\rho\in  \widehat G} 
	   \bigl (\frac{\#G}{ \dim \rho}\bigr )^{-\chi(M)} \prod_{i=1}^k \chi_\rho (g_i)\, ({\text {mod}} \, p) \, ,$$
	   where $\chi_\rho:G\to F$ is the character  of $\rho$ and $p$ is the characteristic of   $F$. It would be interesting to give a similar generalization of \eqref{e2}.    
	  \end{remar} 
 
  \section{The non-orientable case}\label{se4}

A non-oriented version of the Dijkgraaf-Witten invariant  in  dimension $n\geq 1$ can be defined as follows.   Let $G$ be a finite group and     $\alpha\in H^n(G; \ZZ/2\ZZ)$. For a closed connected $n$-dimensional   topological manifold $M$ with fundamental group $\pi$, set 
$$Z_\alpha (M)= (\#G)^{-1}  \sum_{\gamma\in \Hom (\pi , G)} (-1)^{\langle (f_\gamma)^* (\alpha), [M]\rangle} \in \, (\#G)^{-1}\, \ZZ  \, ,$$
where  $f_\gamma:M\to K(G,1) $ is as in Section \ref{se1} and  $\langle (f_\gamma)^* (\alpha), [M]\rangle\in    \ZZ/2\ZZ$ is the value 
of   $  (f_\gamma)^* (\alpha)\in H^n(M; \ZZ/2\ZZ)$ on the fundamental class $ [M] \in H_n(M; \ZZ/2\ZZ)$. For orientable $M$, this is a special case of the definition given in Section \ref{se1}. We therefore restrict ourselves to non-orientable $M$. From now on,   $n=2$.

  We  formulate a Verlinde-type formula for $Z_{\alpha} (M)$. Fix a normalized  2-cocycle  $c$ on $G $   with values in the cyclic group of order two $\{\pm 1\}$.   
  Fix a field $F$ such that $\#G$ is invertible in $F$. 
An irreducible $c$-representation   $\rho:G\to  GL (W)$, where $W$ is a finite dimensional vector space over   $F$, is   {\it self-dual}   if there is a non-degenerate bilinear form $\langle\, , \, \rangle: W\times W\to F$ such that  $\langle \rho(g) (u) , 
\rho(g) (v) \rangle
=\langle u , v \rangle$ for  all $g\in G$ and   $u,v \in W$.   The irreducibility of $\rho$ easily implies   that the form $\langle\, , \, \rangle$ has to be either symmetric or skew-symmetric.  Set $\varepsilon_\rho=  1\in F$ in the former case and $\varepsilon_\rho=  -1\in F$ in the latter case. 
For any non-self-dual  irreducible $c$-representation   $\rho $, set $\varepsilon_\rho=  0\in F$.
  The formula $\rho\mapsto \varepsilon_\rho$ yields a well defined function $\widehat G_c\to \{-1, 0, 1\}\subset F$. This function plays the role of the Frobenius-Schur indicator   in the theory of representations over $\CC$.

Composing  $c:G\times G\to \{\pm 1\}$ with the  isomorphism $\{\pm 1\} \approx  \ZZ/2\ZZ$,   we obtain a $(\ZZ/2\ZZ)$-valued 2-cocycle on $G$. Its    cohomology class in $H^2(G; \ZZ/2\ZZ)$ is denoted $[c]$.

 \begin{theor}\label{th1nnn}  If   $F$ is an algebraically closed field, then for any  normalized 2-cocycle $c:G\times G\to \{\pm 1\}$ and any closed connected non-orientable surface $M$, 
\begin{equation}\label{e2nnn}Z_{[c]} (M)= (\#G)^{-\chi(M)} \, \sum_{\rho\in   \widehat G_c} 
	   (\varepsilon_\rho \dim \rho)^{\chi(M)}   \,.\end{equation}
	 \end{theor}
	 
Note that $\chi(M) \leq 0$ for all  closed connected non-orientable surfaces $M$ distinct from the real   projective plane $P^2$.

    \begin{corol}\label{cor1nnn} Let  $F$ be    a  field of characteristic $p\geq 0$ and $\alpha\in H^2(G; \ZZ/2\ZZ)$. If $p=0$, then  $Z_{\alpha} (M)\in F$ is a non-negative integer for all   closed connected  non-orientable surfaces $M\neq P^2$.  If $p>0$, then  
		  $Z_{\alpha} (M)\in \ZZ/p\ZZ \subset F$   for all   closed connected  non-orientable surfaces $M $.    \end{corol}

The proof  of Theorem \ref{th1nnn} uses    state sum invariants of non-oriented surfaces  introduced by    Karimipour and Mostafazadeh \cite{km}, see also \cite{sn}.  These invariants are derived from so-called $\ast$-algebras. A {\it $\ast$-algebra} over a field $F$ (not necessarily algebraically closed)  is  a finite-dimensional algebra ${\mathcal A}$ over   $F$  endowed   with an $F$-linear involution ${\mathcal A}\to {\mathcal A}, a\mapsto a^\ast$ such that $(ab)^*=b^* a^*$   for all $a, b\in {\mathcal A}$ and  $T(a^*)=T(a)$ for all $a \in {\mathcal A}$, where $T:{\mathcal A}\to F$ is the trace homomorphism defined 
in Section \ref{se2}.  Note that for any $a,b\in {\mathcal A}$, 
$$T(a^*b)= T( (a^* b)^*)= T(b^* a)=   T(ab^*)  \,.$$

A  $\ast$-algebra ${\mathcal A}$ is {\it semisimple} if  
the bilinear form    $T^{(2)}:{\mathcal A}\otimes {\mathcal A} \to F$, defined by $ a\otimes b \mapsto T(ab)$, is   non-degenerate. Consider the vector $v=\sum_i v^1_i \otimes v^2_i\in {\mathcal A}\otimes {\mathcal A}$ satisfying \eqref{plom} and note that
\begin{equation}\label{eddd}
  (\id_{\mathcal A}Ê\otimes \ast) (v)= (\ast \otimes \id )(v)\, .\end{equation}
  Indeed,  for any $a,b \in {\mathcal A}$, 
  $$\sum_i T(a(v^1_i )^*) \,T( bv^2_i)= \sum_i T(   a^* v^1_i ) \,T( bv^2_i)= T(a^*b) =T(ab^*)$$
  $$=
  \sum_i T( a v^1_i  ) \,T( b^*v^2_i)=\sum_i T(  a v^1_i  ) \, T( b (v^2_i)^*  )\,.$$ 
  Now, the non-degeneracy of $T^{(2)}$ implies that $\sum_i (v^1_i)^* \otimes v^2_i=\sum_i v^1_i \otimes (v^2_i)^*$.
  
  A   semisimple  $\ast$-algebra  $({\mathcal A}, \ast)$ over $F$ gives rise to an invariant  of a closed connected (non-oriented) surface $M$ as follows (cf.\  \cite{km}, \cite{sn}). Fix a triangulation of $M$ and an arbitrary orientation  on its 2-simplices. Then proceed as in Section \ref{se2}  with one change: the vector $v_e$ assigned to an edge $e$  is   $v$ if the orientations of two   2-simplices  adjacent to $e$ induce opposite orientations on $e$  and is $ (\id_{\mathcal A}Ê\otimes \ast) (v)= (\ast \otimes \id_{\mathcal A} )(v)$ otherwise. Then   $I_{({\mathcal A},\ast)}(M)=\widetilde T^{(3)} (V)\in F$ is a topological invariant of $M$, where $V=\otimes_e v_e$.  That  $I_{({\mathcal A},\ast)}(M)$ is preserved when   the orientation on a 2-simplex is reversed follows from the formula
$T(abc)= T((abc)^*)=T(c^* b^* a^*)$ for  $a,b,c \in {\mathcal A}$. For orientable $M$, we have  $I_{({\mathcal A},\ast)}(M)=I_{{\mathcal A}} (M)$, where $I_{{\mathcal A}} (M)$ is the invariant of Section \ref{se2}  computed for an arbitrary orientation of $M$. 

The following example was communicated to the author by N.\ Snyder, cf.\ \cite[Theorem 4.2]{sn}.
  Let ${\mathcal A} ={\rm {Mat}}_{d } (F)$   with $d\geq 1$   invertible in $F$ and let $\ast$ be the involution in ${\mathcal A}$ defined by 
$a^*=Q^{-1} a^{\text {Tr}} Q$, where $a\in {\mathcal A}$ and   $Q\in {\rm {Mat}}_{d } (F)$ is an invertible  matrix    such that $Q^{\text {Tr}}= \varepsilon Q$ for   $\varepsilon=\pm 1$. Then the pair $({\mathcal A}, \ast)$ is a semisimple   $\ast$-algebra  and   $I_{({\mathcal A},\ast)} (M)=(\varepsilon d)^{\chi (M)}\cdot 1_F$.

Let    again  $c:G\times G\to \{\pm 1\}$ be a normalized 2-cocycle.  We define a multiplication   $\cdot  $ on the vector space $F[G]$   by $	 g_1\cdot  g_2=  {c(g_1, g_2)} g_1 g_2$ for any  $g_1,g_2\in G$. This turns  $ F[G]$ into an associative unital algebra    $A^{(c)}$  with  involution $\ast$ defined   by $g^\ast  =   {c(g , g^{-1})}\, g^{-1}$ for   $g\in G$. The pair 
$(A^{(c)},\ast)$ is a semisimple $\ast$-algebra. The only non-obvious condition is the equality 
$(ab)^*=b^* a^*$   for  $a, b\in A^{(c)}$. It suffices to check this equality  for $a,b\in G$. It is equivalent then to the five-term  identity
$$ c(ab, (ab)^{-1})=  c(a,  a^{-1}) \,  c(b, b^{-1} ) \, c(a,b)\,   c( b^{-1},  a^{-1}) \, .$$
To check this identity, we substitute $g_1=a, g_2=b, g_3=b^{-1}$ in \eqref{d1} and  obtain  $
c(ab, b^{-1}  )=c(a,b )\,  c(b, b^{-1}  )$. Then   set   $g_1=ab, g_2= b^{-1}, g_3=a^{-1}$ in \eqref{d1} and substitute $
c(ab, b^{-1}  )=c(a,b )\,  c(b, b^{-1}  )$ in the resulting formula. This yields a formula equivalent to the five-term identity.

   \begin{theor}\label{bdbdbdad+}  For  any closed connected surface $M$,   $$Z_{[c]} (M)=(\#G)^{-\chi(M)} \, I_{(A^{(c)},\ast)}(M)\,.$$
\end{theor}

\begin{proof} The proof   is analogous to the proof of Theorem  \ref{bdbdbdad} and we only indicate the main changes. One begins by fixing a triangulation of $M$ and  a total order on the set of the vertices.  As in the proof 
of Theorem  \ref{bdbdbdad}, we define the sets $L(M)$  and   $L_a(M)$  of labelings and admissible labelings  of $M$. Each 2-simplex $\Delta=ABC$ of $M$ with $A<B<C$ is provided with distinguished orientation which induces the direction from $A$ to $B$ on the edge $AB$.    For any   $\ell\in L(M)$,  set   $\ell_1^\Delta= \ell (AB), \ell_2^\Delta= \ell (BC)$, and $ \ell_3^\Delta=\ell (CA)$. Then  
\begin{equation}\label{e491nonor} Z_{[c]} (M)=(\#G)^{-k_0}\sum_{\ell\in L_a(M) } \, \prod_\Delta \, {c(\ell_1^\Delta,\ell_2^\Delta)}  \, .\end{equation}

Let  ${\mathcal A}=A^{(c)}$.  The vector $v\in {\mathcal A}\otimes {\mathcal A}$  is computed by $$v=(\# G)^{-1} \sum_{g\in G}  {c(g,g^{-1})} g \otimes g^{-1}\,  $$
   and                  
   $$(\id_{\mathcal A}  \otimes \ast) (v)=(\ast \otimes \id_{\mathcal A}) (v)=
   (\# G)^{-1} \sum_{g\in G}  g \otimes g  \,  .$$
Let    $\{{\mathcal A}_f\}_{f }$ be a set of   copies of ${\mathcal A}$ numerated by all flags $f$ of $M$. With 
a   labeling $\ell $ of $M$  we associate  a vector  
   $V(\ell)\in \otimes_f \, {\mathcal A}_f$   as in the proof of Theorem  \ref{bdbdbdad}.
   Then $$V= (\#G)^{-k_1}\sum_{\ell \in L(M)}   {\langle \langle c, \ell \rangle\rangle} \, V(\ell)\, ,$$   
for  $$\langle\langle c, \ell \rangle \rangle=\prod_e  {c(\ell(e), \ell(-e))}\in F^*\, ,$$
 where $e$ runs over all  edges of $M$ such that the distinguished orientations of the
 2-simplices adjacent to $e$  induce opposite orientations on $e$.

Set  $U=(\# G)^{-1}\,  T^{(3)}: {\mathcal A}\otimes {\mathcal A} \otimes {\mathcal A} \to F $  and observe that   $$I_{(A^{(c)},\ast)}=\widetilde T^{(3)} (V)=(\#G)^{k_2} \,\widetilde U (V)   = (\#G)^{k_2-k_1} \sum_{\ell \in L(M)} {\langle \langle c, \ell \rangle\rangle}  \,\widetilde U (V(\ell)) $$
  $$= (\#G)^{\chi(M)-k_0} \sum_{\ell \in L_a(M)}{\langle \langle c, \ell \rangle\rangle} \,\widetilde U (V(\ell)) \, .$$
Here  $$\widetilde U (V(\ell))=  
\prod_\Delta  {c(\ell_1^\Delta,\ell_2^\Delta)}  \,    u_\Delta\, ,$$
where $u_\Delta= {c(\ell_3^\Delta ,  (\ell_3^\Delta)^{-1} )}$. If an edge $e$ of $M$ is adjacent to the 2-simplices $\Delta, \Delta'$, then   $e$ contributes   $c(\ell(e), \ell (-e))=\pm 1$ to the product $u_\Delta u_{\Delta'}$ if   the distinguished orientations of $\Delta, \Delta'$  induce opposite orientations on $e$.   Otherwise $e$ contributes $+1$ to $u_\Delta u_{\Delta'}$. Therefore
   $\prod_\Delta u_\Delta  =\langle \langle c, \ell \rangle \rangle$.  The rest of the proof is straightforward. \end{proof}
   
   {\bf Proof of Theorem \ref{th1nnn}}.  To deduce Theorem \ref{th1nnn} from Theorem \ref{bdbdbdad+}, we   split  $A^{(c)}$ as a direct product of matrix algebras. The involution $\ast $ on $A^{(c)}$ induces a permutation $\sigma$ on the set of these algebras. The fixed points of $\sigma$  bijectively correspond to the self-dual irreducible $c$-representations of $G$. The free orbits of $\sigma$ give rise to $\ast$-subalgebras of $A^{(c)}$ of type 
   $B={\rm {Mat}}_{d } (F) \times {\rm {Mat}}_{d } (F)$, where $d$ is invertible in $F$ and the involution 
   $\ast$ on $B$ acts by $(P_1,P_2) \mapsto (P_2^{\text {Tr}}, P_1^{\text {Tr}})$ for   $P_1, P_2 \in 
   {\rm {Mat}}_{d } (F)$. A computation similar to the one in  \cite[Theorem 4.2]{sn}  shows that $I_{ (B, \ast)} (M)=0$. The rest of the argument goes as in the oriented case.

 \begin{remar}    For $c=1$, Formula \eqref{e2nnn}  may be rewritten as      $$ \# \Hom (\pi_1(M) , G)=   \#G \, \, \sum_{\rho\in  \widehat G} 
	   \bigl (\frac{ \varepsilon_\rho \, \#G}{ \dim \rho}\bigr )^{-\chi(M)} \,\, ({\text {mod}} \, p) \, ,$$
	   where $p\geq 0$ is the characteristic of $F$ and $\widehat G$ is the set of irreducible   linear representations of $G$ over $F$   considered up to linear equivalence. 
	   	   \end{remar}

\begin{remar}  Consider in more detail the case $M=P^2$. The group $\pi=\pi_1(P^2)$ is a cyclic group of order 2 so that the homomorphisms $ \pi \to G$ 
are numerated by   elements of the set $S=\{g\in G\, \vert \, g^2=1\}$. Any $g\in S$   gives rise to a (non-homogeneous) generator $[g\vert g]$ of the normalized bar-complex of $G$. This generator is a cycle modulo 2
since $\partial [g \vert g]=  2[g]-[g^2]=0\, ({\text {mod}} \, 2)$. For  $\alpha\in H^2(G;\ZZ/2\ZZ)$ and 
any mapping  $f :P^2\to K(G,1)$, we have
$\langle  f^* (\alpha), [P^2]\rangle =\alpha ([g\vert g])$, where     $g\in G$ is the value of the induced homomorphism $f_{\#}:\pi \to G$  on the non-trivial element of $\pi$ and  $\alpha ([g\vert g])$ is the evaluation of $\alpha$ on the   2-cycle $[g\vert g]$.  Therefore
$$Z_\alpha (P^2)= (\#G)^{-1}  \sum_{g\in S} (-1)^{\alpha ([g\vert g])} \, .$$
If  $\alpha$ is represented by a normalized 
2-cocycle $c:G\times G\to \{\pm 1\} \approx \ZZ/2\ZZ  $, then $$(-1)^{\alpha ([g\vert g])} =c(g,g) \quad  {\text {and}}  \quad Z_{[c]}  (P^2)= (\#G)^{-1}  \sum_{g\in S}  {c(g,g)} \, .$$   Formula \eqref{e2nnn} can now be rewritten
as
$$\sum_{\rho\in   \widehat G_c} 
\varepsilon_\rho \dim \rho= \sum_{g\in S}  {c(g,g)}\,\, ({\text {mod}} \, p) \, ,$$
	   where $p\geq 0$ is the characteristic of $F$.
For $c=1$,  this gives $$\sum_{\rho\in   \widehat G} 
\varepsilon_\rho \dim \rho=\# S\, \,({\text {mod}} \, p)\,.$$ 
	  \end{remar}


\begin{thebibliography}{CJKLS}
\bibitem[Br]{br} K.\ S.\ Brown,   Cohomology of groups. Graduate Texts in Mathematics, 87. Springer-Verlag, New York-Berlin, 1982.     


 \bibitem[DW]{dw}  R.\ Dijkgraaf, E.\ Witten,  
\emph{Topological gauge theories and group cohomology\/},
Comm. Math. Phys. 129 (1990),   393--429. 

\bibitem[Fr]{fr} D.\ Freed,  
\emph{Higher algebraic structures and quantization\/},
Comm. Math. Phys. 159 (1994), no. 2, 343--398. 


\bibitem[FQ]{fq} D.\ Freed,  F.\ Quinn, 
\emph{
Chern-Simons theory with finite gauge group\/},
Comm. Math. Phys. 156 (1993),  435--472. 

 \bibitem[FHK]{fhk}  M.\ Fukuma, S.\ Hosono,  H.\ Kawai, 
\emph{Lattice topological field theory in two dimensions\/},  Comm. Math. Phys.  161  (1994),  no. 1, 157--175.


\bibitem[Jo]{jo} G.\ Jones, \emph{Enumeration of homomorphisms and surface-coverings\/}, 
Quart. J. Math. Oxford Ser. (2) 46 (1995),  485--507. 


\bibitem[KM]{km} V.\ Karimipour, A.\ Mostafazadeh,  
\emph{Lattice topological field theory on nonorientable surfaces\/},  
J. Math. Phys. 38 (1997),   49--66. 

\bibitem[Ka]{ka} G.\ Karpilovsky,   Projective representations of finite groups. Monographs and Textbooks in Pure and Applied Mathematics, 94. Marcel Dekker, Inc., New York, 1985.

 \bibitem[Me1]{me}    A.\ D.\ Mednykh, 
                     \emph{Determination of the number of nonequivalent coverings over a compact Riemann surface\/},
                      Dokl. Akad. Nauk SSSR  239  (1978),   269--271. English translation: Soviet Math.\ Doklady 19 (1978), 318--320.

                          
                    \bibitem[Me2]{me2}      A.\ D.\ Mednykh, 
                     \emph{On the solution of the Hurwitz problem on the number of nonequivalent coverings over a compact Riemann surface\/},   Dokl. Akad. Nauk SSSR  261  (1981),   537--542. English translation: Soviet Math.\ Doklady 24 (1981), 541--545.
		       
		       
		           \bibitem[Sn]{sn}   N.\ Snyder,  
 \emph{Mednykh's Formula via Lattice Topological Quantum Field Theories\/},
math/0703073.

 \bibitem[Tu]{tu} V.\ Turaev, 
\emph{Homotopy field theory in dimension 2 and group-algebras\/},
 math/9910010.
	               \bibitem[Wa]{wa}   M.\ Wakui,  
 \emph{On Dijkgraaf-Witten invariant for $3$-manifolds\/},
 Osaka J. Math.  29  (1992),  no. 4, 675--696.


  
                     \end{thebibliography}
                     \end{document}